\documentclass[12pt,leqno]{amsart}
\usepackage{a4,latexsym,amssymb,amsfonts}
\usepackage{enumerate}
\usepackage{amsthm}
\numberwithin{equation}{section}

\newtheorem{theo}{Theorem}[section]
\newtheorem{lemma}[theo]{Lemma}
\newtheorem{prop}[theo]{Proposition}
\newtheorem{cor}[theo]{Corollary}
\newtheorem{defi}[theo]{Definition}

\theoremstyle{definition}
\newtheorem{rem}[theo]{Remark}

\newtheorem{ass}[theo]{Assumptions}
\def\t{\tau}\def\a{\alpha}

\def\clip{C_{Lip}^1}
\def\vertv{\vert_{_{\V}}}
\def\t{\tau}
\def\a{\alpha}
\def\A{A_0}
\def\V{V^\prime}
\def\ranglev{\rangle_{V,V^\prime}}
\def\p{\pi}
\def\lipd{ ]\kern-1pt_{_{L}}}
\def\lips{[}

\def\e{\varepsilon}
\def\B{\mathcal{B}}

\begin{document}
\title{Equilibrium points for Optimal Investment with Vintage Capital }
\author{Silvia Faggian${}^1$}\footnote{LUM ``Jean Monnet", Casamassima (Bari), I-70010.}
\begin{abstract}
The paper concerns the study of equilibrium points, namely the
stationary solutions to the closed loop equation, of an infinite
dimensional and infinite horizon boundary control problem for
linear partial differential equations. Sufficient conditions for
existence of equilibrium points in the general case are given and
later applied to the economic problem of optimal investment with
vintage capital. Explicit computation of equilibria for the
economic problem in some relevant examples is also provided.
Indeed the challenging issue here is showing that a theoretical
machinery, such as optimal control in infinite dimension, may be
effectively used to {compute} solutions explicitly and easily, and
that the same computation may be straightforwardly repeated in
examples yielding the same abstract structure. No stability result
is instead provided: the work here contained has to be considered
as a first step in the direction of studying the behavior of
optimal controls and trajectories in the long run.\end{abstract}
 \maketitle

\begin{small}
{\bf Subj-class:} Optimization and Control \\{\bf MSC-class:}
49J15, 49J20, 35B37.
\end{small}
\bigskip

\section{Introduction}
The paper concerns the study of equilibrium points of an infinite
dimensional and infinite horizon boundary control problem for
linear partial differential equations. More precisely, we take
into account a state equation of type\begin{equation}\label{eq
stato in H}\begin{cases}y^\prime(\t)=\A y(\t)+Bu(\t), & \t\in
[t,+\infty)\\
y(t)=x\in H,\end{cases}\end{equation} where $H$ is the state
space, $y:[t,+\infty)\to H$ is the trajectory, $U$ is the control
space and $u:[t,+\infty)\to U$ is the control, $A_0:D(A_0)\subset
H\to H$ is the infinitesimal generator of a strongly continuous
semigroup of linear operators $\{e^{\t \A}\}_{\t\ge0}$ on $H$,
 and the control operator $B$ is linear and {\it unbounded}, say
 $B:U\to[D(A_0^*)]^\prime$.
Besides,  we consider a cost functional given by
\begin{equation}\label{J in H}J_\infty(t,x,u)
=\int_t^{+\infty}e^{-\lambda
\tau}\left[g_0\left(y(\t)\right)+h_0\left(u(\t)\right)\right]d\t\end{equation}
 where the functions  $g_0$ and $h_0$ are convex functions as better
specified later.

\medskip
More precisely, by \emph{equilibrium points} we mean stationary
solutions to the closed loop equation associated by Dynamic
Programming to (\ref{eq stato in H}) that is
\begin{equation}\label{CLEi} y(\t)=e^{(\t-t)A}x+\int_t^\t
e^{(\t-\sigma)A}B(h_0^*)'(-B^*\Psi^\prime(y(s))) d\sigma,\ \
\tau\in[t,+\infty[ ,\
\end{equation}
where $h_0^*$ is the convex conjugate of $h_0$,  $\Psi$ is the
value function of the optimal control problem for initial time
$t=0$, more precisely
$$\Psi(x)=Z_\infty(0,x)=\inf_{u\in L^p_\lambda(0,+\infty;U)}J_\infty(0,x,u),$$
and
$$G(x)=(h_0^*)'(-B^*\Psi^\prime(x)),$$
is the unique optimal feedback map, as shown in \cite{FaGo2,Fa6}.
Indeed the problem of minimizing $J_\infty(t,x,u)$ with respect to
$u$ over the Banach space
$$L_\lambda^p(t,+\infty;U)=\{u:[t,+\infty)\to U\ :\ \tau\mapsto u(\tau)e^{-\frac{\lambda}{p}\tau}\in
L^p(t,+\infty;U)\},\  p\ge2,$$ was studied paper by Faggian and
Gozzi in \cite{FaGo2}, and by Faggian in \cite{Fa6} by means of
Dynamic Programming methods, deriving:\begin{itemize}

\item existence and uniqueness for the associated
Hamilton-Jacobi-Bellman (briefly, HJB) equation

\item a feedback formula for optimal controls in terms of the
spatial gradient of the value function,

\item Pontryagin Maximum Principle.\end{itemize} All of these
results are recalled in Section \ref{prel}.

As a first result here, we give sufficient conditions for
existence of equilibrium points in the general case, we apply such
results to the problem optimal investment with vintage capital
described in Section 3 (cfr. Section 4). Nevertheless the most
interesting result of the paper is the explicit computation of
equilibria for the economic problem in some relevant examples
(Section 5). Indeed the challenging issue here is showing that a
theoretical machinery such as optimal control in infinite
dimension may be effectively used to \emph{compute} solutions
explicitly and easily, and that the same computation may be
straightforwardly repeated in examples yielding the same abstract
structure.

No stability result is instead provided. Under this respect, the
work here contained has to be considered as a first step in the
direction of studying the behavior of optimal controls and
trajectories in the long run.

\subsection{Bibliographical notes}
It is well known that control problems with unbounded control
operator $B$ arise when we rephrase into abstract terms some
boundary control problem for PDEs (or, more generally, problems
with control on a subdomain). Indeed, we motivate our framework
with the application to the economic problem of optimal investment
with vintage capital  in the framework by Barucci and Gozzi
\cite{BG1} \cite{BG2} that we describe in detail in Section
\ref{esemp}. Similar problems  with unbounded control operator
have been discussed in a series of papers by this author and
others. The unconstrained case has been studied both in the case
of finite and infinite horizon \cite{Fa2,Fa3,FaGo2} while
\cite{FaGo} contains the finite horizon case with constrained
controls. The (finite horizon) case with both boundary control and
state constraints is treated in \cite{Fa4}. The case of infinite
horizon (without constraints) has been treated in \cite{FaGo2} and
\cite{Fa6}.

Some further references on \emph{boundary control} in infinite
dimension follow. We recall that such problems have been studied
in the framework of classical/strong solutions and in that of
viscosity solutions. Regarding Dynamic Programming in the
classical/strong framework, the available results mainly regard
the case of linear systems and quadratic costs (where HJB reduces
to the operator Riccati equation). The reader is then referred
{\it e.g.} to the book by Lasiecka and Triggiani \cite{LT}, to the
book by Bensoussan, Da Prato, Delfour and Mitter \cite{BDDM}, and,
for the case of nonautonomous systems, to the papers by
Acquistapace, Flandoli and Terreni \cite{AFT, AT1, AT2, AT3}. For
the case of a linear system and a general convex cost, we mention
the papers by this author \cite{Fa1,Fa0,Fa2,Fa3}. On Pontryagin
maximum principle for boundary control problems we mention again
the book by Barbu and Precupanu (Chapter 4 in \cite{BP}).

For viscosity solutions and HJB equations in infinite dimension we
mention the series of papers by Crandall and Lions \cite{CL} where
also some boundary control problem arises. Moreover, for boundary
control we mention Gozzi, Cannarsa and Soner \cite{CGS} and the
paper by Cannarsa and Tessitore \cite{CT} on existence and
uniqueness of viscosity solutions of HJB.  We note also that a
verification theorem in the case of viscosity solutions has been
proved in some finite dimensional case in the book by Yong and
Zhou \cite{YZ}. We finally mention the paper by Fabbri
\cite{Fabbri.viscoso} where the author derives an existence and
uniqueness result for the viscosity solution of HJB associated to
optimal investment with vintage capital (with infinite horizon and
without constraints), that is the application of Section 3 of the
present paper, obtaining the results by making use of the specific
properties of the state equation, while no result is there
provided for the general problem.

We mention also some fundamental papers and books on the case of
\emph{distributed control} in the classical/strong framework such
as the works by Barbu and Da Prato \cite{BD1, BD2,BD3} for some
linear convex problems,
 to Di Blasio \cite{D1,D2} for the case
of constrained control, to Cannarsa and Di Blasio \cite{CD} for
the case of state constraints, to Barbu, Da Prato and Popa
\cite{BDP} and to Gozzi \cite{G1,G2,G3} for semilinear systems.

Regarding applications, on control on a subdomain (boundary or
point control) we refer the reader to the many examples contained
in the books by Lasiecka and Triggiani \cite{LT},
 and by Bensoussan {\it et al} \cite{BDDM}. Moreover,
for economic models with vintage capital the reader may see
 the papers  by Barucci and Gozzi \cite{BG1}, \cite{BG2},  the papers by
 Feichtinger, Hartl, Kort, Veliov et al.
\cite{F1,F2,F3,FHS}, and for population dynamic
 the book by Iannelli \cite{I},   the paper by Ani\c ta,
Iannelli, Kim and Park \cite{AIKP}, and the papers by Almeder,
Caulkins, Feichtinger, Tragler, and Veliov \cite{Almeder} and
references therein.

\section{Preliminaries}\label{prel}
We here recall all the relevant results that are needed in the
sequel. The reader may find the proofs of all statements in
\cite{FaGo2,Fa6}. According to the notation there contained, if
$X$ and $Y$ are Banach spaces, we denote by $\vert\cdot\vert_X$
the norm on $X$, by $\vert\cdot\vert$ the euclidean norm in
$\mathbb{R}$, and we set
\begin{equation*}\begin{split}
&Lip(X;Y)=\{f:X\to Y ~:~\lips f\lipd:=\sup_{x,y\in X,~x\neq y}
\frac{\vert f(x)-f(y)\vert_{Y}}{\vert x-y\vert_X} <+\infty\}\\
&\clip(X):=\{f\in C^1(X)~:~ \lips f^\prime\lipd<+\infty\}\\
&\B_r(X,Y):=\{f:X\to Y~:~\vert f\vert_{\B_r}:=\sup_{x\in X} {\vert
f(x)\vert_Y\over 1+\vert x\vert_X^r}<+\infty\},\ \ \
\B_r(X):=\B_r(X,\mathbb{R}).\\
\end{split}\end{equation*}
Moreover we set
\begin{equation*}
\Sigma_0(X):=\{w\in \B_2(X)\ :\ w\ {\rm is\ convex,\ }
w\in\clip(X) \}\end{equation*} and, for $T>0$
\begin{equation*}\begin{split}\mathcal{Y}([0,T]\times X)=
\{w:[0,T]&\times X\to \mathbb{R}\ :\ w\in C([0,T],\B_2(X)),\
\\w(t,\cdot)\in&\Sigma_0(X),\ \forall t\in[0,T], \ \ w_x\in C([0,T], \B_1(X))\}\\
\end{split}
\end{equation*}
All the spatial derivatives above have to be intended as Frech\'et
differentials.
\bigskip

\noindent Then we consider two Hilbert spaces $V,\V$, being dual
spaces, which we do not identify for reasons which are recalled in
Remark \ref{noidentif} and we denote the duality pairing by
$\langle\cdot,\cdot\rangle$. We set $\V$ as the state space of the
problem, and denote with $U$ the control space, being $U$ another
Hilbert space. The state space is $\V$ and the control space is
$U$. For any fixed $x$ in $\V$ and $t>0$ and $\tau\ge t$, the
solution to the state equation in  $\V$ is given by variation of
constant formula by
\begin{equation}\label{sevc}y(\t)=e^{(\t-t)A}x+\int_t^\t
e^{(\t-\sigma)A}Bu(\sigma)d\sigma,\ \ \tau\in[t,+\infty[ ,\
\end{equation}
while the target functional is of type
\begin{equation}\label{J-t-T}J_\infty(t,x,u):=\int_t^{+\infty}e^{-\lambda
\tau}[g_0(y(\tau))+h_0(u(\tau))]d\tau.\end{equation} We assume the
following hypotheses hold:

\begin{ass}\label{asst2}
\begin{enumerate}

\item[1.] $A:D(A)\subset\V\to\V$ is the infinitesimal generator of
a strongly continuous semigroup $\{e^{\t A}\}_{\t\ge0}$ on $\V$;

\item[2.] $B\in L(U,\V)$;

\item[3.] there exists $\omega\ge0$ such that $\vert e^{\t
A}x\vertv\le  e^{\omega \t}\vert x\vertv,~\forall \t\ge0$;

\item[4.]  $g_0, \phi_0\in\Sigma_0(\V)$

\item[5.] $h_0$ is  convex, lower semi--continuous, $\partial_u
h_0$  is  injective.

\item[6.] $h_0^*(0)=0$, $h_0^*\in \Sigma_0(V)$; \item[7.] $\exists
a>0$,  $\exists b\in \mathbb{R}$, $\exists p\ge2$ : $h_0(u)\ge
a\vert u\vert_U^p+b$, $\forall u\in U$;

Moreover, either \item[8.a] $p> 2$, $\lambda>2\omega$.

 or
 \item[8.b] $\lambda>\omega$,  and $g_0,\phi_0 \in \B_1(\V).$
\end{enumerate}\end{ass}
\begin{rem}\label{noidentif} We do not identify $V$ and $\V$ for
 in the applications the problem is naturally set in a Hilbert space $H$, such that
$V\subset H\equiv H^\prime\subset\V$ (with all bounded
inclusions). Indeed, in order to avoid the discontinuities due to
the presence of $B$, as they appear in (\ref{eq stato in
H})(\ref{J in H}), we work in the extended state space $\V$
related to $H$ in the following way: $V$ is the Hilbert space
$D(\A^*)$ endowed with the scalar product $(v|w)_V:=
(v|w)_H+(\A^*v|\A^*w)_H$, $\V$ is the dual space of $V$ endowed
with the operator norm. Then assume that $B\in L(U,\V)$, and
extend
 the semigroup $\{e^{tA_0}\}_{t\ge0}$ on $H$ to a semigroup
$\{e^{tA}\}_{t\ge0}$ on the space ${V^\prime}$, having
infinitesimal generator $A$, a proper extension of $A_0$. The
reader is referred to \cite{Fa3} for  a detailed treatment. The
coefficient $\omega$ could be any real number, but
  is assumed positive in order to avoid double
 proofs for positive and negative signs.
 \hfill\qed \end{rem}

\begin{rem} \label{estensioni1} Note that the functions
$g$ and $\varphi$ arising from applications usually appear to be
defined and $C^1$ on $H$, not on the larger space $\V$. Then, we
here need to {\it assume} that they can be extended to {\it
$C^1$-regular functions on $\V$} - which is a non trivial issue.
We refer the reader to \cite{Fa2}, \cite{Fa3} and \cite{FaGo2} for
a thorough discussion on this issue.\hfill\qed\end{rem} The
functional $J_\infty(t;x,u)$ has to be minimized with respect to
$u$ over the set of admissible controls
\begin{equation}L^p_\lambda(t,+\infty;U)=\{u:[t,+\infty)\to U
\ ;\ \t\mapsto u(\t)e^{-\frac{\lambda \t}{p}}\in
L^p(t,+\infty;U)\},
\end{equation}
which is Banach space with the norm $$
\|u\|_{L^p_\lambda(t,+\infty;U)}=\int_t^{+\infty} |u(\tau)|_U^p
e^{-\lambda \tau}d\tau = \|e^{-\frac{\lambda (\cdot)
}{p}}u\|_{L^p(t,+\infty;U)}.
$$

 The value function is then defined as
$$Z_\infty(t,x)=\inf_{u\in L^p_\lambda(t,+\infty;U)}J_\infty(t,x,u).$$  As it is easy to
check that
$$Z_\infty(t,x)=e^{-\lambda t}Z_\infty(0,x)$$
one may associate to the problem the following stationary HJB
equation
\begin{equation}\label{SHJB}-\lambda \psi(x)+\langle
\psi^\prime(x), A
x\rangle-h_0^*(-B^*\psi^\prime(x))+g(x)=0,\end{equation} whose
candidate solution is the function $Z_\infty(0,\cdot)$.

\medskip

\noindent We will use the following definition of solution for
equation (\ref{SHJB}).

\begin{defi}\label{defsolSHJB}
A function $\psi$ is a classical solution of the stationary HJB
equation (\ref{SHJB}) if it belongs to $\Sigma_0 (\V)$ and
satisfies (\ref{SHJB}) for every $x\in D(A)$.
\end{defi}


\begin{theo}\label{th:mainnew} Let Assumptions \ref{asst2} hold.
Then there exists a unique classical solution $\Psi$ to
$(\ref{SHJB})$ and it is given by the value function of the
optimal control problem, that is
$$\Psi(x)=Z_\infty(0,x)=\inf_{u\in L^p_\lambda(0,+\infty;U)}J_\infty(0,x,u).$$
\end{theo}

Once we have established that $\Psi$ is the classical solution to
the stationary HJB equation, and that it is differentiable, we can
build optimal feedbacks and prove the following theorem.

\begin{theo}\label{th:uniquefeedback}
Let Assumptions \ref{asst2} hold. Let $t\ge 0$ and $x\in \V$ be
fixed. Then there exists a unique optimal pair $(u^*,y^*)$. The
optimal state $y^*$ is the unique solution of the Closed Loop
Equation
\begin{equation}\label{CLE} y(\t)=e^{(\t-t)A}x+\int_t^\t
e^{(\t-\sigma)A}B(h_0^*)'(-B^*\Psi^\prime(y(s))) d\sigma,\ \
\tau\in[t,+\infty[ .\
\end{equation}
while the optimal control $u^*$ is given by the feedback formula
$$
u^*(s) = (h_0^*)'(-B^*\Psi^\prime(y^*(s))).
$$
\end{theo}

Let $x\in \V$ and $t\ge0$ be fixed, and consider the dual system
associated to (\ref{eq stato in H}), that is
\begin{equation}\label{cost}\p^\prime(\tau)=(\lambda -A_0^*)
\p(\tau)-g_0^\prime(y(\tau)), \quad \
\tau\in[t,+\infty)\end{equation} where $\p:[t,+\infty)\to V$ (the
dual variable, or co-state of the system)  is the unknown, and
$y=y(\cdot;t,x, u)$ is the trajectory starting at $x$ at time $t$
and driven by control $u$, given by (\ref{sevc}). We assume such
equation is also subject to the following transversality condition
\begin{equation}\label{tc}
\lim_{T\to+\infty}\p(T)=0.\end{equation} We denote any solution of
(\ref{cost})(\ref{tc}) also by $\p(\cdot;t,x,u)$ or by
$\p(\cdot;t,x)$ to remark its dependence on the data.

\begin{defi} Let Assumptions \ref{asst2} [1-7] be satisfied.
We define the mild solution of (\ref{cost})(\ref{tc}) as the
function $\p:[t,+\infty)\to V$ given by
\begin{equation}\label{pmild} \p(\tau)=\int_\tau^{+\infty}
e^{(A_0^*-\lambda)(\sigma-\tau)}g_0^\prime(y(\sigma))d\sigma.\end{equation}
\end{defi}

In the sequel we show that such definition is natural.

\begin{lemma}\label{p^*def} Let Assumption \ref{asst2} [1-7] be satisfied,
and assume $p\ge2$ and $\lambda>2\omega$. Then $\p$ given by
(\ref{pmild}) is well defined and belongs to $C^0(t,+\infty;V)$.

Moreover:

$(i)$ if $p> 2$ then $\p\in L^q_\lambda(t,+\infty;V);$

$(ii)$ if $p=2$ then $\p\in L^2_{\lambda+\e}(t,+\infty;V)\cap
L^2(t,T;V),\ \forall \ T<+\infty, \ \e>0.$ \end{lemma}

\begin{theo} If $\p\in W^{1,1}(t,+\infty; V)$ satisfies $(\ref{cost})$ almost everywhere in
$[t,+\infty)$ and  $(\ref{tc})$ then $\p$ is given by
(\ref{pmild}), that is $\pi$ is the mild solution of
$(\ref{cost})(\ref{tc})$.
\end{theo}

\begin{rem} Assume $p\ge 2$, $\lambda>2\omega$. Then:
\begin{itemize}\item $\lambda\le\omega p$ implies
$y\in L^r(t,+\infty;\V)$ for all $r<\frac{\lambda}{p}$, and $y\in
L^{\frac{\lambda}{p}}(t,T;\V)$ for all $T<+\infty$; \item
$\lambda>\omega p$ implies $y\in L^r(t,+\infty;\V)$ for all
$r<{p}$, and $y\in L^p(t,T;\V)$ for all $T<+\infty$.
\end{itemize}\end{rem}

\begin{defi} Let Assumption \ref{asst2} [1-7] be satisfied,
and assume $p\ge2$ and $\lambda>2\omega$. Let also $T>t$ be either
finite or $+\infty$. We say that a given pair $(u,y)\in
L^p_\lambda(t,T;U)\times L^1_{loc}(t,T;\V)$ is {\rm extremal} if
and only if there exists a function $\p\in L^q_\lambda(t,T;V)$
satisfying in mild sense, along with $u$ and $y$, the following
set of equations
$$y^\prime(\tau)=Ay(\tau)+Bu(\tau),\ \tau\in[ t,T);\ \ y(t)=x $$
$$\p^\prime(\tau)=(\lambda -A_0^*) \p(\tau)-g_0^\prime(y(\tau)),
 \ \tau\in[ t,T);$$
 $$\lim_{s\to+\infty}\p(s)=0, when \ T=+\infty;\ \ \ \p(T)=0,\ when \ T<+\infty$$
 \begin{equation}\label{mp}-B^* \p(\tau)\in \partial h_0(u(\tau)),
\ for \ a. a.\ \tau\in[t,T).\end{equation}
\end{defi}

\begin{theo}\label{Pmp} {\rm \textbf{(Maximum Principle).}}
Let Assumptions \ref{asst2} [1-7] be satisfied, $\lambda>2\omega$.
Then, for all $p\ge2$ and $T<+\infty$, the couple $(u^*,y^*)$ is
optimal at $(t,x)$ - for the problem of minimizing (\ref{eq stato
in H})(\ref{J-t-T}) - if and only if it is extremal.
\end{theo}

\begin{theo} \label{p^*} Let $(u^*,y^*)$ be optimal at $(0,x)$ and let $\p^*(\cdot;0,x)$
be the associated co-state. Then
$$\Psi^\prime(x)=\p^*(0;0,x).$$
Consequently,
$$\Psi^\prime(y^*(\tau))=\p^*(\tau;\tau,y^*(\tau)).$$
\end{theo}

\bigskip

\section{The motivating example}\label{esemp}
We here describe our motivating example: the infinite horizon
problem of optimal investment with vintage capital, in the setting
introduced by Barucci and Gozzi \cite{BG1}\cite{BG2}, and later
reprised and generalized by Feichtinger et al. \cite{F1,F2,F3},
and by
  Faggian \cite{Fa2,Fa3}.

The capital accumulation is described by the following system
\begin{equation}\label{ipde}\begin{cases}&\frac{\partial y( \tau
,s)}{\partial \tau }+
 \frac{\partial y( \tau ,s)}{\partial s}+
\mu y( \tau ,s) =u_1( \tau ,s), \quad (\tau,s)
\in]t,+\infty[\times
]0,\bar s]\\
& y( \tau ,0)=u_0( \tau ), \quad \tau \in]t,+\infty[\\
&y(t,s)=x(s), \quad s\in [0,\bar s]\end{cases}\end{equation} with
$t>0$ the initial time, $\bar s\in[0,+\infty]$ the maximal allowed
age, and $\tau\in[0,T[$ with  horizon $T=+\infty$. The unknown
$y(\tau,s)$ represents the amount of capital goods of age $s$
accumulated at time $\tau$, the initial datum is a function $x\in
L^2(0,\bar s)$, $\mu>0$ is a depreciation factor. Moreover,
$u_0:[t,+\infty[\to {{\mathbb R}}$ is the investment in new
capital goods ($u_0$ is  the boundary  control) while
$u_1:[t,+\infty[\times[0,\bar s]\to {{\mathbb R}}$ is the
investment at time $\tau$ in capital goods of age $s$ (hence, the
distributed control). Investments are jointly referred to as the
control $u=(u_0,u_1)$.

Besides, we consider the firm profits represented by the
functional
$$I(t,x;u_0,u_1 )=\int_t^{+\infty}e^{-\lambda\tau} [R({Q(\tau)})
-c({u(\tau)})]d\tau$$ where, for some given positive measurable
coefficient $\a$, we have that
$${Q(\tau)}=\int_0^{\bar
s}\alpha(s){{y(\tau,s)}}ds$$ is the output rate (linear in
$y(\tau)$) $R$ is a concave revenue from $Q(\tau)$ (i.e., from
$y(\tau)$). Moreover we have
$$c({u_0(\tau),u_1(\tau)})=\int_0^{\bar{s}}
  c_1(s,u_1(\tau,s))ds +c_0(u_0(\tau)),$$
  with $c_1$ indicating the investment cost rate for
  technologies of age $s$, $c_0$  the investment cost in new
  technologies, including adjustment-innovation, $c_0$, $c_1$ convex in the control variables.

The entrepreneur's problem is that of maximizing $I( t, x;u_0,u_1
\kern-1pt)$ over all state--control pairs $\{y, (\kern-1pt
u_0,u_1\kern-1pt)\kern-1pt\}$ which are solutions in a suitable
sense of equation (\ref{ipde}). Such problems are known as {\em
vintage capital} problems, for the capital goods depend jointly on
time $\tau$ and on age $s$, which is equivalent to their
dependence from time and  vintage $\tau-s$.

\bigskip
When rephrased in an infinite dimensional setting, with
$H:=L^2(0,\bar s)$ as state space, the state equation (\ref{ipde})
can be reformulated as a linear control system with an unbounded
control operator, that is

\begin{equation}\begin{cases}\label{eq:statoecinH}
y^{\prime }(\tau)=A_0 y(\tau)+Bu(\tau), &\tau\in]t,+\infty[;\\
y(t)=x,\end{cases}
\end{equation}
where $y:[t,+\infty[\to H$, $x\in H$, $A_0:D(A_0)\subset H\to H$
is the infinitesimal generator of a strongly continuous semigroup
$\{ e^{\A t}\}_{t\ge0}$ on $H$ with domain $D(A_0)=\{f\in
H^1(0,\bar s):f(0)=0\}$ and defined as $A_0 f(s)=-f^\prime(s)-\mu
f(s)$, the control space is $U={{\mathbb R}}\times H$, the control
function is a couple $u\equiv(u_0,u_1):[t,+\infty[\to {{\mathbb
R}}\times H$, and the control operator is given by $Bu\equiv
B(u_0,u_1)= u_1+u_0\delta _{0}$, for all $(u_0,u_1)\in{{\mathbb
R}}\times H$, $\delta_0$ being the Dirac delta at the point $0$.
Note that, although $B\not\in L(U,H)$, is $B\in
L(U,D(\A^*)^\prime)$. The reader can find in \cite{BG1} the
(simple) proof of the following theorem, which we will exploit in
a short while.
\begin{theo} \label{th:statoinH} Given any initial datum $x\in H$ and control
$u \in L^p_\lambda (t,+\infty;U)$ the mild solution of the
equation (\ref{eq:statoecinH})
$$
y(s)=e^{(s-t)A}x+\int_t^s e^{(s-\tau)A}Bu(\tau)d\tau
$$
belongs to $C([t,+\infty);H)$.
\end{theo}

Following Remark \ref{noidentif}, we then set $$V=D(A_0^*)=\{f\in
H^1(0,\bar s):f(\bar s)=0\}$$ and $\V =D(\A^*)^\prime$. Regarding
the target functional, we set
$$J_\infty(t,x;u):=-I(t,x;u_0,u_1),$$
with:

 $g_0:\V\to
{{\mathbb R}}, ~g_0(x) =- R( \langle\alpha , x\rangle),$


$h_0: U\to {{\mathbb R}}, ~h_0(u)=c_0(u_0)+\int_0^{\bar{s}} c_1(s,
u_1(s)) ds.$

\begin{rem} Here the extension of the datum
 $g_0$ to $\V$
is straightforward, as long as we assume that $\alpha\in V$ and
replace scalar product in $H$ with the duality in $V,\V$.

Note further that $\omega=0$, $\lambda>0$ (the type of the
semigroup is negative and equal to $-\mu$).  \hfill\qed\end{rem}

As the problem now fits into our abstract setting, the main
results of the previous sections apply to the economic problem
when data $R$, $c_0$, $c_1$ satisfy Assumption \ref{asst2}[8.a] or
[8.b]. In particular, such thing happens in the following two
interesting cases:

\begin{itemize}
    \item     If we assume, for instance, that
$R$ is a concave, $C^1$, sublinear function (for example one could
take $R$ quadratic in a bounded set and then take its linear
continuation, see e.g. \cite{F1,F3}), and $c_0$, $c_1$ quadratic
functions of the control variable, then Assumption
\ref{asst2}[8.b] holds.
    \item Assumption \ref{asst2}[8.a] is instead satisfied when $R$ is, for
instance, quadratic - as it occurs in some other meaningful
economic problems - and $c_0$, $c_1$ are equal to $+\infty$
outside some compact interval, and equal to any convex {\it
l.s.c.} function otherwise. Such case corresponds to that of
constrained controls (controls that violate the constrain yield
infinite costs).
\end{itemize}


In these cases, Theorems  \ref{th:mainnew},
\ref{th:uniquefeedback} hold true. In particular Theorem
\ref{th:uniquefeedback} states the existence of a unique optimal
pair $(u^*,y^*)$ for any initial datum $x \in \V$. Note that in
general the optimal trajectory $y^*$ lives in $\V$. However, since
the economic problem makes sense in $H$, we would now like to
infer that  whenever   $x$ (the initial age distribution of
capital) lies in $H$, then   the whole optimal trajectory lives in
$H$. Indeed, this is guaranteed by Theorem \ref{th:statoinH}.

All these results allow to perform the analysis of the behavior of
the optimal pairs and to study phenomena such as the diffusion of
new technologies (see e.g. \cite{BG1,BG2}) and the anticipation
effects (see e.g. \cite{F1,F3}). With respect to the results in
\cite{BG1,BG2}, here also the case of nonlinear $R$ (which is
particularly interesting from the economic point of view, as it
takes into account the case of large investors) is considered.
With respect to the results in \cite{F1,F3}, here the existence of
optimal feedbacks yields a tool to study more deeply the long run
behavior of the trajectories, like the presence of long run
equilibrium points and their properties.

\bigskip

\section{Equilibrium points}
We call \emph{equilibrium point} any stationary solution of the
closed loop equation
$$y^\prime(\tau)=Ay(\tau)+B(h_0^*)^\prime(-B^*\Psi^\prime(y(\tau))), $$
that is any $x\in\V$ such that
\begin{equation}\label{equili1}Ax+B(h_0^*)^\prime(-B^*\Psi^\prime(x))=0.\end{equation}
\begin{lemma} Let  Assumptions \ref{asst2} [1-7] be satisfied, $p\ge2$, $\lambda\ge2\omega$.
Any equilibrium point $x\in D(A)$ satisfies
$$Ax+B(h_0^*)^\prime(-B^*(\lambda-A_0^*)^{-1}g_0^\prime(x))=0.$$\end{lemma}
\begin{proof}
Let $\bar x$ be a solution to (\ref{equili1}). Then there exist a
stationary solution $\bar p$ to the co-state equation given by
 $$\bar p:=p^*(\tau;\bar x)=\int_\tau^\infty e^{(A_0^*-\lambda)(\sigma-\tau)}
 g_0^\prime(\bar x)d\sigma\equiv\int_0^\infty e^{(A_0^*-\lambda)r}
 g_0^\prime(\bar x)dr=(\lambda-A_0^*)^{-1}g_0^\prime(\bar x),$$
 where the last equality holds by definition of  $(\lambda-A_0^*)^{-1}$.
\end{proof}

\bigskip


  Whenever $A$ proves invertible, then, equilibrium points can be
regarded also as fixed points of the operator $T:\V\to \V$,
defined by
\begin{equation}\label{T}
Tx:=-A^{-1}B(h_0^*)^\prime(-B^*\Psi^\prime(x)),\end{equation}
where $\Psi$ is the unique strong solution of the HJB equation
(\ref{SHJB}), or searched among fixed point of
\begin{equation}\label{Tlambda}
Tx:=-A^{-1}B(h_0^*)^\prime(-B^*(\lambda-A_0^*)^{-1}g_0^\prime(
x)),\end{equation}

\begin{lemma}\label{fixedpoint} Let Assumptions \ref{asst2} [1-7] be satisfied and assume
that
$${\lambda-\omega}>\Vert (A_0^*)^{-1}\Vert_{L(H)}\Vert B\Vert^2_{L(\V,U)}
[(h_0^*)^\prime][g_0^\prime].$$ Then there exists a unique
solution $\bar x\in D(A)$ to (\ref{equil1}).
\end{lemma}

To prove the assertion we need the following result.

\begin{prop} Assume that $0\in \rho(A_0^*)$
(that  is, $(A_0^*)^{-1}$ is well defined and bounded in $H$).
Then $A^{-1}$ has bounded inverse on $\V$,
 defined by the position $$
\langle A^{-1}f,\varphi\ranglev= \langle
f,(A_0^*)^{-1}\varphi\ranglev, \ for\ all\ f\in \V\ and\
\varphi\in V.$$ Moreover
$$\Vert A^{-1}\Vert_{L(\V)}\le \Vert (A_0^*)^{-1}\Vert_{L(H)}.$$
\end{prop}

\begin{proof}  For all
$f\in \V$ and $\varphi\in V$ we have
\begin{equation}\begin{split}
\vert \langle A^{-1}f,\varphi\ranglev\vert&=\vert \langle
f,(A_0^*)^{-1}\varphi\ranglev\vert\\
&\le\vert f\vertv\vert(A_0^*)^{-1}\varphi\vert_V\\
&=\vert f\vertv(\vert(A_0^*)^{-1}\varphi\vert_H+\vert A_0^*(A_0^*)^{-1}\varphi\vert_H)\\
&=\vert f\vertv(\vert(A_0^*)^{-1}\varphi\vert_H+\vert(A_0^*)^{-1}A_0^*\varphi\vert_H)\\
&\le\vert f\vertv\Vert(A_0^*)^{-1}\Vert_{L(H)}\vert\varphi\vert_V
\end{split}\end{equation}\end{proof}

  \begin{proof} (Lemma \ref{fixedpoint}) Define $T: \V\to \V$ as
as in (\ref{T}).  Then $T$ is Lipschitz continuous, with Lipschitz
constant smaller than 1. Indeed, according to the content of the
proof of Lemma 4.8 in \cite{FaGo2} (with null final cost)
$$[\Psi^\prime]\le\frac{[g_0^\prime]}{\lambda-\omega}$$ so that
\begin{equation}\begin{split}
\vert Tx-Ty\vertv&\le \Vert (A_0^*)^{-1}\Vert_{L(H)}\Vert
B\Vert^2_{L(\V,U)}
[(h_0^*)^\prime]\frac{[g_0^\prime]}{\lambda-\omega}\vert
x-y\vertv.\end{split}\end{equation}
\end{proof}

 In the particular case of the economic problem described in
 Section 3 and with the data there defined, the preceding result
 reads as follows.

\begin{cor} If 
$$\lambda+\mu>\frac{1}{\mu}\ [(h_0^*)^\prime]\
[R^\prime]\ \vert \alpha\vert^2_V,\ \  or\ \
\lambda+\mu>\frac{\bar s}{\sqrt 2}\ [(h_0^*)^\prime]\ [R^\prime]\
\vert \alpha\vert^2_V,$$ then there exists a unique equilibrium
point for optimal investment with vintage capital described in
Section 3.
\end{cor}

 \begin{proof} In the case of the economic problem we have
 $\Vert B\Vert\le1$, $\omega=-\mu$. By means of Hille-Yosida
 Theorem
$$\Vert (A_0^*)^{-1}\Vert_{L(H)}\le \frac{1}{\mu}.$$
If we show that moreover
$$\Vert (A_0^*)^{-1}\Vert_{L(H)}\le \frac{\bar s}{\sqrt{2}},$$
the rest of the proof is straightforward. Recall that
$D(A_0^*)=\{f\in H^1(0,\bar s)\ :\ f(\bar s)=0\}$, moreover
$$(A_0^*)^{-1}f(s)=-\int_s^{\bar s}e^{-\mu(\sigma-s)}f(\sigma)d\sigma.$$
By means of H\"older inequality one derives
\begin{equation}\begin{split}\vert
(A_0^*)^{-1}f\vert_H^2&=\int_0^{\bar s}\left\vert
\int_s^{\bar s}e^{-\mu(\sigma-s)}f(\sigma)d\sigma\right\vert^2ds\\
&\le\frac{1}{2\mu}\int_0^{\bar s}(1-e^{-2\mu(\bar s-s)})\vert
f\vert^2_{L^2(s,\bar s)}ds\\
&\le\vert f\vert_H^2\frac{1}{2\mu}\left(\bar
s-\frac{1-e^{-2\mu\bar
s}}{2\mu}\right)\\
&\le \vert f\vert_H^2\frac{\bar s^2}{2},\end{split}\end{equation}
Instead if  $\mu=0$ we have $$\vert
(A_0^*)^{-1}f\vert_H^2\le\int_0^{\bar s}\left(\frac{\bar
s}{2}-\frac{(\bar s-\sigma)^2}{2}\right)\vert f(\sigma)\vert^2
d\sigma\le \frac{\bar s^2}{2}\vert f\vert_H^2.$$
\end{proof}
\begin{rem} If
$$c(u)=\int_0^{\bar
s}[\beta_1(s)u_1^2(s)+q_1(s)u_1(s)]ds+[\beta_0u_0^2+q_0u_0] ,\ \
R(Q)=bQ-Q^2,$$  the preceding are implied respectively by
$${\lambda}+\mu>\frac{1}{\mu}\;a\;\left(1+\frac{1}{\beta_0}\right)\left(\int_0^{\bar
s}\alpha(s)^2ds+\int_0^{\bar
s}\alpha^\prime(s)^2ds-\mu\alpha(0)^2\right)$$ and
$$\lambda+\mu>\frac{\bar s}{\sqrt 2}\;a\;
\left(1+\frac{1}{\beta_0}\right)\left(\int_0^{\bar
s}\alpha(s)^2ds+\int_0^{\bar
s}\alpha^\prime(s)^2ds-\mu\alpha(0)^2\right).$$ The statement
derives from
$$[g_0']=2a\vert\alpha\vert^2_V,\
[(h_0^*)^\prime]=\Vert
M_{\frac{1}{2\beta}}\Vert_{L(U)}\le\frac{1}{2}\left(1+\frac{1}{\beta_0}\right),$$
where $M_{\frac{1}{2\beta}}$ is the operator described in
(\ref{M}) in the next section.\end{rem}
\bigskip

\section{Explicit formulae for equilibrium points for vintage capital with convex-linear cost}
Throughout the subsection we assume that $h_0$ is given by
 \begin{equation}\label{ipo2}\begin{split}h_0(u)&=(M_\beta u\vert u)_U+(q\vert u)_U\\
&=\int_0^{\bar
s}u_1(s)[\beta_1(s)u_1(s)+q_1(s)]ds+u_0[\beta_0u_0+q_0]\end{split}\end{equation}
where $\beta=(\beta_0,\beta_1)\in R\times L^\infty(0,\bar s)$,
with $\beta_1(s),\beta_0\ge\epsilon\ge0$, $q=(q_0,q_1)\in R\times
L^2(0,\bar s)\equiv U$, and $M_\beta:U\to U$ is given by
\begin{equation}\label{M}
M_\beta u(s):=(\beta_0 u_0,\beta_1(s)u(s)).\end{equation} Then is
is easy to show that
\begin{equation}\label{ipo3}\begin{split}
h_0^*(u)&=(M_{\frac{1}{4\beta}}(u-q)\vert
u-q)_U,\\
&=\int_0^{\bar
s}\frac{1}{4\beta_1(s)}[u_1(s)-q_1(s)]^2ds+\frac{1}{4\beta_0}[u_0-q_0]^2\\
\end{split}\end{equation}so that
$$(h_0^*)^\prime(u)=M_{\frac{1}{2\beta}}(u-q),$$
more explicitly
$$(h_0^*)^\prime(u)(s)=\left(\frac{1}{2\beta_0}[u_0-q_0];\frac{1}{2\beta_1(s)}[u_1(s)
-q_1(\cdot)]\right).$$

\begin{lemma}\label{equil1} Let $(\ref{ipo2})$ be satisfied,
and $R\in C^1(\mathbb{R})$, with $R'$ Lipschitz-continuous.
 Moreover, we set
$$w_1=-A^{-1}BM_{\frac{1}{2\beta}}B^*(\lambda-A_0^*)^{-1}\alpha,
\ \ and\ \ w_2=A^{-1}BM_{\frac{1}{2\beta}}q.$$ Then there exists
an equilibrium point $ x\in H$ if an only if there exists a real
number $ \eta$ satisfying
$$\eta=R^\prime(\langle \alpha, w_2+\eta w_1\rangle).$$
In that case $$x=w_2+\eta w_1.$$ Moreover, if
$R^{\prime\prime}\le0$, then such equilibrium point does exist and
is unique.
\end{lemma}

\begin{rem}Note that $w_1$ and $w_2$ may be explicitly computed.
According to the notation in \cite{BG1}, we set
$$\bar\alpha(s)=(\lambda-A_0^*)^{-1}\alpha(s)=\int_s^{\bar s}e^{-(\mu+\lambda)(\sigma-s)}
\alpha(\sigma) d\sigma$$ the discounted return associated with a
unit of capital of vintage s, and see that
\begin{equation}\begin{split}w_1(s)&=-[A^{-1}BMB^*(\lambda-A_0^*)^{-1}\alpha](s)\\
&=-[A^{-1}B M(\bar\alpha(0),\bar\alpha)](s)\\
 &=-\left[A^{-1}B\left(\frac{\bar\alpha(0)}{2\beta_0},\frac{\bar\alpha(\cdot)}{2\beta_1(\cdot)}\right)\right](s)\\
 &=-\frac{\bar\alpha(0)}{2\beta_0}[A^{-1}\delta_0](s)-
 [A^{-1}\frac{\bar\alpha(\cdot)}{2\beta_1(\cdot)}](s)\\
&=\frac{\bar\alpha(0)}{2\beta_0}e^{-\mu
s}+\int_0^se^{-\mu(s-\sigma)}\frac{\bar\alpha(\sigma)}
{2\beta_1(\sigma)}d\sigma.\end{split}\end{equation} Similarly, one
shows that $$w_2(s)=\frac{q_0}{2\beta_0}e^{-\mu
s}+\int_0^se^{-\mu(s-\sigma)}\frac{q_1(\sigma)}
{2\beta_1(\sigma)}d\sigma.$$
\end{rem}
\begin{proof} In the case of the economic problem, $D(A)=H$, so that the operator
$T:H\to H$ defined in (\ref{Tlambda}) is given by
\begin{equation}\begin{split}
Tx&=-A^{-1}BM(-B^*\bar p-q)\\
&= -R^\prime(\langle \alpha, x\rangle) A^{-1}BMB^*(\lambda-A_0^*)^{-1}\alpha+A^{-1}BMq\\
&=  R^\prime(\langle \alpha,
x\rangle)w_1+w_2.\end{split}\end{equation} Hence
$$Tx=x\iff
x-w_2=R^\prime(\langle \alpha, x\rangle)w_1$$ which is true if and
only if $x-w_2=\eta w_1$ for some  $\eta\in \mathbb{R}$, that is
if and only if
\begin{equation}\label{puntieq}
\eta=R^\prime(\langle \alpha, w_2+\eta w_1\rangle).\end{equation}
The last assertion is straightforward, as $\langle
w_1,\alpha\rangle\ge0$.
\end{proof}

From the preceding Lemma one derives the following results.

\begin{lemma} \label{theo:eq1} Let assumptions
$(\ref{ipo2})$ be satisfied,  and set $c_1:=\langle
\alpha,w_1\rangle,\ c_2:=\langle \alpha,w_2\rangle.$ Then there
exists a unique equilibrium point $\bar x$ in each of the
following cases:
\begin{enumerate} \item[(i)] If $R(Q)=-aQ^2+bQ$,  then
$$\bar x=w_2-\frac{2ac_2-b}{1+2ac_1} w_1;$$ \item[(ii)] If
$R(Q)=\ln(1+Q),$ for $Q\ge0$ and $R(Q)=Q$ for $Q<0$, then
$$\bar x=w_2+\frac{\sqrt{(1+c_2)^2+4c_1}-(1+c_2)}{2c_1}\; w_1$$ \item[(iii)]If
$R(Q)=(1+Q)^\gamma-1,$ with $\gamma\in(0,1)$, for $Q\ge0$ and
$R(Q)=\gamma Q$ for $Q<0$, then $\bar x=w_1+\bar \eta w_2$ where
$\bar \eta$ is the unique positive solution of
$$\eta=\frac{\gamma}{(1+c_1\eta+c_2)^{1-\gamma}}.$$\end{enumerate}
\end{lemma}


\begin{rem} Results similar to those contained in
Lemma \ref{equil1}, and \ref{theo:eq1} may be proved also in the
case $h_0(u)=\vert u\vert_U^p$, $p\ge2$.\end{rem}

\end{document}